\documentclass{amsart}
\usepackage[english]{babel}
\usepackage{amsthm}
\usepackage{inputenc}
\usepackage{amssymb}
\usepackage{graphicx}
\usepackage{color}
\newtheorem{thm}{Theorem}[section]

\newtheorem{prop}[thm]{Proposition}
\newtheorem{defn}[thm]{Definition}

\numberwithin{equation}{section}

\def\ll{{L}}
\def\bc{\begin{center}}
\def\ec{\end{center}}
\def\d{\displaystyle}

\begin{document}

\title[Solvable Leibniz algebras with filiform nilradical]{Solvable Leibniz algebras with filiform nilradical}%

\author{L.M. Camacho,
B.A. Omirov and K.K. Masutova}
\address{[L.M. Camacho] Dpto. Matem\'{a}tica Aplicada I.
Universidad de Sevilla. Avda. Reina Mercedes, s/n. 41012 Sevilla.
(Spain)} \email{lcamacho@us.es}
\address{[B.A. Omirov, K.K. Masutova] Institute of Mathematics (Uzbekistan), 29, F.Hodjaev srt., 100125, Tashkent (Uzbekistan)}
\email{omirovb@mail.ru, kamilyam81@mail.ru}
%

\begin{abstract}
In this paper we continue the description of solvable Leibniz algebras whose nilradical is a filiform algebra. In fact, solvable Leibniz algebras whose nilradical is a naturally graded filiform Leibniz algebra are described in \cite{Campo} and \cite{SolGFil}. Here we extend the description to solvable Leibniz algebras whose nilradical is a filiform algebra. We establish that solvable Leibniz algebras with filiform Lie nilradical are Lie algebras.
\end{abstract}
\maketitle


\section{Introduction}

Leibniz algebras were introduced by Loday \cite{loday2} as a non-skew symmetric version of Lie algebras. These algebras generalize Lie algebras in natural way. The theory of Leibniz algebras has been actively investigated in the last two decades. Many results of the theory of Lie algebras have been extended to Leibniz algebras. For instance, the classical results on Cartan subalgebras \cite{Cartan}, Levi's decomposition \cite{barnesLevi}, the properties of solvable algebras with given nilradical \cite{SolNul} and other from the theory of Lie algebras are also true for Leibniz algebras \cite{ayuom,barnesEngel,patsourakos}.

As far as physical applications are concerned, we note that solvable Lie algebras often occur as Lie algebras of symmetry groups of differential equations \cite{olver}. Group invariant solutions can be obtained by symmetry reduction, using the subalgebras of the symmetry algebra \cite{winternitz}. In this procedure an important step is to identify the symmetry algebra and its subalgebras as abstract Lie algebras. A detailed identification presupposes the existence of a classification of Lie algebras into isomorphism classes.

In this paper we continue the description of solvable algebras with a given nilradical. The first work was devoted to the description of such Lie algebras under some condition (see \cite{Malcev}) the complemented space to nilradical forms abelian subalgebra, consisting of semisimple elements of the algebra. However, the structure of nilradical depends on this subalgebra. Later, Mubarakzjanov G.M. proposed the description of solvable Lie algebras with a given structure of nilradical \cite{muba} by means of outer derivations of the nilradical. Papers \cite{ancochea1,ancochea2,Campo,ndogmo,tremblay,wang} were devoted to the application of Mubarakzjanov's method for solvable Lie algebras with different kinds of nilradicals. Some results of the Lie algebra theory generalized to Leibniz algebras in \cite{ayuom} allow us to apply the Mubarakzjanov's method to the case of Leibniz algebras.  In this direction papers \cite{SolNul} and \cite{SolGFil} deal with the description of solvable Leibniz algebras with null-filiform and naturally graded filiform nilradicals, respectively.

The aim of the present paper is to classify solvable Leibniz algebras with filiform nilradical. Thanks to papers \cite{vigoYu} and \cite{abror}, we already have the classification of filiform Leibniz algebras. It should be noted that the description of Lie algebras with filiform nilradicals obtained in this work is a new result for Lie algebras, as well.

In section 2 the necessary definitions and results for understanding the main parts of the paper are given. In section 3 the classification of solvable Leibniz algebras with filiform non-Lie Leibniz nilradicals is obtained. The last section deals with the description of solvable Leibniz algebras with filiform Lie nilradical.

Throughout the paper we consider finite-dimensional vector spaces and algebras over the field of the complex numbers. Moreover, in the multiplication table of an algebra omitted products are assumed to be zero and if it is not noted, we shall
consider non-nilpotent solvable algebras.

\section{Preliminaries}

In this section we give the necessary definitions and preliminary results on outer derivations of nilradicals of Leibniz algebras and the descriptions of filiform Leibniz algebras.

\begin{defn} An algebra $(L,[-,-])$ over a field $F$ is called a Leibniz algebra if for any $x,y,z\in L$, the so-called Leibniz identity
\[ \big[[x,y],z\big]=\big[[x,z],y\big]+\big[x,[y,z]\big] \] holds.
\end{defn}

From the Leibniz identity we conclude that for any $x, y\in\ll$ the elements $[x,x]$ and $[x,y]+[y,x]$ lie in  the right annihilator of the algebra $\ll$ (denoted by $Ann_r(\ll)=\{x\in \ll:\ [y,x]=0, \ \mbox{ for all }y\in\ll\}$).

\begin{defn} A linear map $d:\ll \rightarrow \ll$ of a Leibniz algebra $\ll$ is said to be a derivation if for any $x, y\in \ll,$ the following condition holds:
$$d([x,y])=[d(x),y]+[x,d(y)].$$
\end{defn}

For a given element $x$ of a Leibniz algebra $\ll$ the operator of right multiplication ${R}_x:\ll \rightarrow \ll$, defined as ${R}_x(y)=[y,x]$ for $y\in \ll$, is a derivation. This kind of derivations are called \textit{inner derivations}.

Any Leibniz algebra $\ll$ is associated with the algebra of right multiplications ${R}(\ll)=\{{R}_x | \ x\in \ll\}$, which endowed with a structure of Lie algebra by means of the bracket $[{R}_x,{R}_y]={R}_x{R}_y-{R}_y {R}_x.$ The equality
$[{R}_x,{R}_y]={R}_{[y,x]}$ holds true because of Leibniz identity and an antisymmetric isomorphism between ${R}(\ll)$ and the quotient algebra $\ll/Ann_r(\ll)$ is established.

\begin{defn}\cite{muba}
Let $d_1,$ $d_2,$ $\dots,$ $d_n$ be derivations of a Leibniz algebra $\ll.$ The derivations
$d_1,$ $d_2,$ $\dots,$ $d_n$ are said to be nil-independent if
$$\alpha_1 d_1+\alpha_2 d_2+\cdots +\alpha_n d_n$$
is not nilpotent for any scalars $\alpha_1,\alpha_2,\dots,\alpha_n \in F.$

In other words, if for any $\alpha_1,\alpha_2,\dots,\alpha_n \in F$ there exists a natural number $k$ such that
$(\alpha_1 d_1+\alpha_2 d_2+\cdots +\alpha_n d_n)^k=0,$ then $\alpha_1=\alpha_2=\cdots=\alpha_n=0.$
\end{defn}

For a Leibniz algebra $\ll,$ the sequences of two-sided ideals defined recursively as follows:
$$ L^1=L,\qquad L^{k+1}=[L^k,L^1], \qquad k \geq 1,\quad \qquad \ll^{[1]}=\ll,\qquad \ll^{[s+1]}=[\ll^{[s]}, \ll^{[s]}],\quad s\geq 1 $$ are said to be {\it the lower central} and {\it the derived series of $\ll,$} respectively.

\begin{defn} A Leibniz algebra $L$ is called
nilpotent (respectively, solvable) if there exists  $m\in\mathbb N $ ($t\in \mathbb{N}$) such that $L^{m-1}\neq 0$ and $L^m=0$ (respectively, $\ll^{[t-1]}\neq 0$ and $\ll^{[t]}= 0$).
The minimal number $m$ (respectively, $t$) with such property is said to be the index of nilpotency (respectively, solvability) of the algebra $\ll.$
\end{defn}

Evidently, the index of nilpotency of an $n$-dimensional algebra is not greater than $n+1.$

Since the sum of two nilpotent ideals is nilpotent ideal, we can consider the maximal nilpotent ideal of a Leibniz algebra. The maximal nilpotent ideal is called {\it nilradical}. Notice that the nilradical is not the radical in the sense of Kurosh, because the quotient Leibniz algebra by its nilradical may contain a nilpotent ideal (see \cite{jacobson}).

\begin{defn} An $n$-dimensional Leibniz algebra $L$ is said to be filiform if
$\dim L^i=n-i$ for $2\leq i \leq n$.
\end{defn}

Let $R$ be a solvable Leibniz algebra. Then it can be decomposed in the form $R=N\oplus Q,$ where $N$ is the nilradical and $Q$ is the complementary vector space. Since the square of a solvable Leibniz algebra is contained into nilradical \cite{ayuom}, we get the nilpotency of the ideal $R^2$ and consequently, $Q^2\subseteq N.$

\begin{thm}\cite{SolNul} \label{thm26}
Let R be a solvable Leibniz algebra and $N$ be its nilradical. Then the dimension of the complementary vector space to N is not greater than the maximal number of nil-independent derivations of N.
\end{thm}

A nilpotent Leibniz algebra is called \textit{characteristically nilpotent} if all its derivations are nilpotent.
If the nilradical $N$ of a Leibniz algebra is a characteristically nilpotent then, according to Theorem \ref{thm26}, a solvable Leibniz algebra is nilpotent. Therefore, we shall consider solvable Leibniz algebras with non-characteristically nilpotent filiform nilradicals. The paper is divided into two parts, the former deals with non-characteristically nilpotent filiform non-Lie Leibniz nilradicals and the latter one we describe solvable Leibniz algebras with non-characteristically nilpotent filiform Lie nilradicals.

\begin{thm}\label{yu}\cite{vigoYu}
Let $\mathfrak{g}$ be an $(n+1)$-dimensional non-characteristically nilpotent filiform Lie algebra. Then, it is isomorphic to one of the following non-isomorphic algebras:
$$\begin{array}{ll}
L_n:&\qquad \quad \quad Q_n,\ n-odd:\\[5mm]
\left\{\begin{array}{ll}
[e_0,e_i]=e_{i+1},&1\leq i\leq n-1;
\end{array}\right.& \qquad \quad \quad
\left\{\begin{array}{ll}
[e_0,e_i]=e_{i+1},&1\leq i\leq n-2,\\{}
[e_i,e_{n-i}]=(-1)^{i}e_n,&1\leq i\leq n-1;
\end{array}\right.
\end{array}$$
$$\begin{array}{l}
A_{n+1}^r(\alpha_1,\dots,\alpha_t),\ 1\leq r\leq n-3,\ t=\lfloor\frac{n-r-1}{2}\rfloor:\\[5mm]
\begin{array}{ll}
[e_0,e_i]=e_{i+1},& 1\leq i\leq n-1,\\{}
[e_i,e_j]=\left(\sum\limits_{k=i}^t(-1)^{k-i}\alpha_k
\left (\begin{array}{c}
j-k-1\\
k-i
\end{array} \right )\right) e_{i+j+r},
 & 1\leq i<j\leq n-2, \ i+j+r\leq n;
 \end{array}\\[10mm]

 B_{n+1}^r(\alpha_1,\dots,\alpha_t),\ 1\leq r\leq n-3,\ t=\lfloor\frac{n-r-2}{2}\rfloor,\ n-odd:\\[5mm]

\begin{array}{ll}
[e_0,e_i]=e_{i+1},& 1\leq i\leq n-2,\\{}
[e_i,e_{n-i}]=(-1)^ie_n, & 1\leq i\leq n-1,\\{}
[e_i,e_j]=\left(\sum\limits_{k=i}^t(-1)^{k-i}\alpha_k
\left (\begin{array}{c}
j-k-1\\
k-i
\end{array} \right )
\right)e_{i+j+r}, & 1\leq i,j\leq n-1, \ i+j+r\leq n-1,
\end{array}
 \end{array}$$
where the parameters $(\alpha_1,\dots,\alpha_t)$ satisfy the polynomial relations emanating from the Jacobi identity and at least one parameter $\alpha_i\neq 0.$
\end{thm}

In the following theorem, all $(n+1)$-dimensional filiform Leibniz algebras decomposes into three families of algebras.

\begin{thm}\cite{OmRa} \label{thm28}
Any complex $(n+1)$-dimensional filiform Leibniz algebra admits a basis $\{e_0,e_1,\dots,e_n\}$ such that the table of multiplication of the algebra has one of the following
forms:
$$\begin{array}{rl}
F_1(\alpha_3,\alpha_4,\dots,\alpha_n,\theta):&\left\{\begin{array}{ll}
[e_0,e_0]=e_2,&\\{}
[e_i,e_0]=e_{i+1},&1\leq i\leq n-1,\\{}
[e_0,e_1]=\displaystyle\sum_{k=3}^{n-1} \alpha_k e_k+\theta e_n,&\\{}
[e_i,e_1]=\displaystyle\sum_{k=i+2}^n \alpha_{k+1-i} e_k,&1\leq i\leq n-2;
\end{array}\right.\\[15mm]

F_2(\beta_3,\beta_4,\dots,\beta_n,\gamma):&\left\{\begin{array}{ll}
[e_0,e_0]=e_2,&\\{}
[e_i,e_0]=e_{i+1},&2\leq i\leq n-1,\\{}
[e_0,e_1]=\displaystyle\sum_{k=3}^{n} \beta_k e_k,&\\{}
[e_i,e_1]=\displaystyle\sum_{k=i+2}^n \beta_{k+1-i} e_k,&2\leq i\leq n-2;
\end{array}\right.\\[15mm]

F_3(\theta_1,\theta_2,\theta_3):&\left\{\begin{array}{ll}
[e_i,e_0]=e_{i+1},& 1\leq i\leq n-1,\\{}
[e_0,e_i]=-e_{i+1},&2\leq i\leq n-1,\\{}
[e_0,e_0]=\theta_1 e_n,&\\{}
[e_0,e_1]=-e_2+\theta_2 e_n,&\\{}
[e_1,e_1]=\theta_3 e_n,&\\{}
[e_i,e_j]=-[e_j,e_i]\in lin <e_{i+j+1},e_{i+j+2},\dots,e_n>,&1\leq i\leq n-2,\\{}
& 2\leq j\leq n-i,\\{}
[e_i,e_{n-i}]=-[e_{n-i},e_i]=\alpha (-1)^{i} e_n ,&1\leq i\leq n-1,
\end{array}\right.
\end{array}$$
where $\alpha\in \{0,1\}$ for odd $n$ and $\alpha=0$ for even $n$. Moreover, the structure constants of an algebra from $F_3(\theta_1,\theta_2,\theta_3)$ should satisfy the Leibniz identity.
\end{thm}

It is easy to see that algebras of the first and the second families are non-Lie algebras. Moreover, a Leibniz algebra of the third family is a Lie algebra if and only if $(\theta_1,\theta_2,\theta_3)=(0,0,0).$

From the list of Theorem \ref{thm28} we only indicate non-characteristically nilpotent filiform non-Lie Leibniz algebras.

\begin{thm}\cite{abror} \label{abror}
An arbitrary non-characteristically nilpotent filiform non-Lie Leibniz algebra is isomorphic to one of the following non-isomorphic algebras:

\begin{itemize}
\item $\begin{array}{l}
F_1(0,0,\dots,0,1)\ \mbox{ and }\
F_1^s(\alpha_3,\alpha_4,\dots,\alpha_{n-1},\alpha_n,\alpha_n),\ 3\leq s\leq n,
\end{array}$
where

$\alpha_k=\left\{\begin{array}{ll}
0,&k\not \equiv s\ \mbox{\rm mod (s-2) }\\[1mm]
(-1)^t C_{t+1}^{s-1},&k\equiv s\ \mbox{\rm mod (s-2) }
\end{array}\right.$
and $t=\d\frac{k-s}{s-2},$ $ 3\leq k\leq n$ and $C_n^p$ is the p-th Catalan number,\\

\item

\begin{itemize}
\item[$\checkmark $] for even $n$:

 $F_{2}(0,0,\dots,0,0,1)$ and $F_2^j(0,0,\dots,0,\underbrace{1}_j,0,\dots,0,0,0)$ with $n\geq 4$ and $3\leq j\leq n,$

\item[$ \checkmark $] for odd $n$:

 $F_{2}^1(0,0,\dots,0,\beta_{\frac{n+2}{2}},0,\dots,0,0,1)$ and $F_2^j(0,0,\dots,0,\underbrace{1}_j,0,\dots,0,0,0)$ with $n\geq 4$ and $3\leq j\leq n,$
\end{itemize}

\

\item $F_3(1,0,0),\quad F_3(0,1,0), \quad F_3(0,0,1).$
\end{itemize}
\end{thm}

\section{Solvable Leibniz algebras with filiform non-Lie Leibniz nilradical}

In this section we investigate solvable Leibniz algebras whose nilradical is one from the list of Theorem \ref{abror}. In order to demonstrate the considering cases for three families of Theorem \ref{thm28}, we divide this section into three subsections.

\subsection{Solvable Leibniz algebras with nilradical a non-characteristically nilpotent algebra of the family $F_1(\alpha_3,\alpha_4,\dots,\alpha_n,\theta)$.}

\

\

It follows the matrix form of any derivation of an algebra of the family
$F_1(\alpha_3, \alpha_4, \dots, \alpha_n, \theta )$.

\begin{prop} \cite{abror}\label{pr31} Any derivation of a filiform Leibniz algebras from the family $F_1( \alpha_3, \alpha_4, \dots, \alpha_n, \theta )$ has the following matrix form:
\[\begin{pmatrix}
 a_0&a_1&a_2&a_3&\dots&a_{n-2}&a_{n-1}&a_n\\
0&a_0+a_1&a_2&a_3&\dots&a_{n-2}&b_{n-1}&b_n\\
 0& 0& 2a_0+a_1&a_2+a_1\alpha_3&\dots&a_{n-3}+a_1\alpha_{n-2}&
a_{n-2}+a_1\alpha_{n-1}&a_{n-1}+a_1\alpha_n\\
0& 0& 0& 3a_0+a_1& \dots&a_{n-4}+2a_1\alpha_{n-2}&
a_{n-3}+2a_1\alpha_{n-1}&a_{n-2}+2a_1\alpha_{n-1}\\
 \vdots& \vdots& \vdots& \vdots&\ddots& \vdots&\vdots &\vdots \\
0& 0& 0& 0&\dots& (n-2)a_0+a_1&a_2+(n-3)a_1\alpha_3 & a_3+ (n-3)a_1\alpha_4\\
 0& 0& 0& 0&\dots& 0&(n-1)a_0+a_1 & a_2+(n-2)a_1\alpha_3\\
 0& 0& 0& 0&\dots& 0& 0& na_0+a_1
\end{pmatrix},\]
where
\[a_0(\theta - \alpha_n) = 0, \quad a_1(\alpha_n - \theta) =
a_{n-1} - b_{n-1}, \quad
 \alpha_3(a_1-a_0)=0,\]
\[
\alpha_k(a_1-(k-2)a_0)=\frac k 2 a_1\sum\limits_{j=4}^{k}\alpha_{j-1}\alpha_{k-j+3}, \quad 4 \leq k \leq n.\]
\end{prop}

\

{\bf Case $F_1(0,0,\dots,0,1)$.}

\

From Proposition \ref{pr31} we conclude that the number of nil-independent outer derivations of algebra
$F_1(0,0,\dots,0,1)$ is equal to one. Thus, we have that any solvable Leibniz algebra whose nilradical is $F_1(0,0,\dots,0,1)$ has dimension $n+2.$

\begin{prop}
There are not any $(n+2)$-dimensional solvable Leibniz algebras with nilradical $F_1(0,0,\dots,0,1)$.
\end{prop}

\begin{proof}
Let $\ll$ be a solvable Leibniz algebra satisfying the condition of the proposition. We complement the basis  $\{e_0,e_1,\dots,e_n\}$ of nilradical $F_1(0,0,\dots,0,1)$ by a basis element $x$ of $Q$.

From the table of multiplication of $F_1(0,0,\dots,0,1)$ we conclude that $<e_2,e_3,\dots, e_n>\subseteq Ann_r(\ll).$
Using Proposition \ref{pr31}, we derive the following products in the algebra $\ll:$
$$
\begin{array}{lll}
[e_0,e_0]=e_2,&\\{}
[e_i,e_0]=e_{i+1},& 1\leq i\leq n-1,&\\{}
[e_0,e_1]=e_n,& &\\{}
[e_0,x]=\d\sum_{i=1}^n a_i e_i,& &[x,e_0]=\d\sum_{i=0}^n\beta_i e_i,\\{}
[e_1,x]=\d\sum_{i=1}^{n-2}a_i e_i+(a_{n-1}+a_1)e_{n-1}+b_n e_n,& & [x,e_1]=\d\sum_{i=0}^n\gamma_i e_i\\{}
[e_i,x]=\d\sum_{k=i}^{n}a_{k-i+1} e_k,&2\leq i\leq n,&[x,x]=\d\sum_{i=0}^n\delta_i e_i.
\end{array}
$$

Taking the change as follows:
$$x'=x-\d\sum_{i=2}^{n-1} \beta_{i+1}e_i,$$
we can assume $[x,e_0]=\beta_0 e_0+\beta_1 e_1+\beta_2 e_2.$

The equalities
$$0=[e_0,[e_0,x]+[x,e_0]]=[e_0,[x,x]]=[e_0,[e_1,x]+[x,e_1]]$$ imply
$$\beta_0=0, \ \beta_1=-a_1, \ \delta_0=\delta_1=0, \ \gamma_0=\gamma_1=0.$$

Considering
$$0=[x,e_2]=-a_1e_2+\sum_{i=3}^{n}\gamma_i e_i,$$
we obtain $a_1=0.$ If we substitute $a_1=0$ in the relations of Proposition \ref{pr31}, we have $a_0=0$. Consequently,
the restriction of the operator $R_x$ to nilradical $F_1(0,0,\dots,0,1)$ is a nilpotent derivation. Therefore, we get a contradiction with the existence of any solvable Leibniz algebra with nilradical $F_1(0,0,\dots,0,1).$
\end{proof}

{\bf Case $F_1^s(\alpha_3,\alpha_4,\dots,\alpha_{n-1},\alpha_n,\alpha_n)$.}

\

Let us fix the first non zero parameter $\alpha_s\neq 0$ of the algebra $F_1^s(\alpha_3,\alpha_4,\dots,\alpha_{n-1},\alpha_n,\alpha_n)$. Then, from the relations of Proposition \ref{pr31}, we deduce $a_1=(s-2)a_0$ and $b_{n-1}=a_{n-1}.$ Therefore, the number of nil-independent outer derivations of nilradical $F_1^s(\alpha_3,\alpha_4,\dots,\alpha_{n-1},\alpha_n,\alpha_n)$ is equal to one.

\begin{prop}
There are not any $(n+2)$-dimensional solvable Leibniz algebras with nilradical $F_1^s(\alpha_3,\alpha_4,\dots,\alpha_{n-1},\alpha_n,\alpha_n)$.
 \end{prop}

\begin{proof}
Let $\ll$ be a solvable Leibniz algebra with nilradical $F_1^s(\alpha_3,\alpha_4,\dots,\alpha_{n-1},\alpha_n,\alpha_n)$.
Since in the general form of a non-nilpotent derivation of nilradical $F_1^s(\alpha_3,\alpha_4,\dots,\alpha_{n-1},\alpha_n,\alpha_n)$ the parameter $a_0\neq 0$ (otherwise due to equality  $a_1=(s-2)a_0$ a derivation is nilpotent), without loss of generality, one can assume $a_0=1$.

Since for a basis element of the space $Q$ the general form of the derivation $R_x$ is presented in Proposition \ref{pr31}, we have the following multiplications:
$$\begin{array}{lll}
[e_0,e_0]=e_{2},& [e_i,e_{0}]=e_{i+1}, & 1\leq i\leq n-1,\\{}
[e_0,e_1]=\d\sum_{k=3}^{n}\alpha_k e_k,&
[e_i,e_{1}]=\d\sum_{k=i+2}^n \alpha_{k+1-i}e_k,& 1\leq i\leq n-2,\\{}
[e_0,x]=e_0+(s-2)e_1+\d\sum_{i=2}^n a_i e_i,&  [x,e_0]=\d\sum_{i=0}^{n}\beta_i e_i,&\\{}
[e_1,x]=(s-1)e_1+\d\sum_{i=2}^{n-1} a_i e_i+b_n e_n,&  [x,e_1]=\d\sum_{i=0}^n \gamma_i e_i,&
\end{array}$$
$$\begin{array}{ll}
[e_i,x]=(s-2+i) e_i+\d\sum_{j=i+1}^n (a_{j+1-i}+(i-1)(s-2)\alpha_{j-i+2}) e_j, & \quad \quad 2\leq i\leq n,\\{}
[x,x]=\d\sum_{i=0}^n \delta_i e_i.&
\end{array}$$

Evidently, $<e_2, e_3,\dots,e_n>\subseteq Ann_r(\ll)$, consequently, $[x,e_i]=0$ with $2\leq i\leq n$.

The equality  $[e_0,[x,x]]=0$ implies $\delta_0=\delta_1=0.$

Making the following change of basis:
$$x'=x-\d\sum_{i=2}^{n-1}\beta_{i+1}e_i$$
we obtain
$$\begin{array}{ll}
[x',e_0]=\beta_0 e_0+\beta_1 e_1+\beta_2 e_2,&[e_0,x']=[e_0,x]=e_0+(s-2)e_1+\d\sum_{i=2}^n a_i e_i,\\{}
[e_1,x']=[e_1,x]=(s-1)e_1+\d\sum_{i=2}^{n-1} a_i e_i+b_n e_n,&[x',e_1]=\d\sum_{i=0}^n \gamma'_i e_i,\\{}
[x',x']=\d\sum_{i=0}^n \delta'_i e_i.&
\end{array}$$

From the equalities
$$[e_0,[e_0,x]+[x,e_0]]=[e_0,[e_1,x]+[x,e_1]]=0$$ we conclude $$\beta_0=-1, \ \beta_1=-s+2, \  \gamma_0=0, \ \gamma_1=-(s-1).$$

A contradiction obtained from $0=[x,e_2]=[x,[e_1,e_0]]=-(s-1)e_2+\d\sum_{k=3}^n \Delta_k e_k$ with $s\geq 3$ completes the proof of the proposition.
\end{proof}


\

\subsection{Solvable Leibniz algebras with nilradical a non-characteristically nilpotent algebra of the family $F_2(\beta_3,\beta_4,\dots,\beta_n,\gamma)$.}

\

\

In this subsection we consider  the family of algebras $F_2(\beta_3,\beta_4,
\dots, \beta_n, \gamma )$. Similar to the above subsection, firstly we
describe the derivations of such algebras.

\begin{prop}\cite{abror}\label{pr34} Any derivation of a filiform Leibniz
algebra of the family $F_2(\beta_3,\beta_4, \dots, \beta_n, \gamma)$ has the following matrix form:
\[\begin{pmatrix}
a_0&a_1&a_2&a_3&\dots&a_{n-2}&a_{n-1}&a_n\\
0&b_1&0&0&\dots&0&-a_1\gamma & b_n\\
 0& 0& 2a_0&a_2+a_1\beta_3&\dots&a_{n-3}+a_1\beta_{n-2}&
a_{n-2}+a_1\beta_{n-1}&a_{n-1}+a_1\beta_n\\
 0& 0& 0& 3a_0& \dots& a_{n-4}+2a_1\beta_{n-3}& a_{n-3}+2a_1\beta_{n-2}&a_{n-2}+2a_1\beta_{n-1}\\
\vdots& \vdots& \vdots& \vdots&\ddots& \vdots&\vdots &\vdots \\
0& 0& 0& 0&\dots& (n-2)a_0&a_2+(n-3)a_1\beta_3 & a_3+(n-3)a_1\beta_4\\
 0& 0& 0& 0&\dots& 0&(n-1)a_0 & a_2+(n-2)a_1\beta_3\\
0& 0& 0& 0&\dots& 0& 0& na_0
\end{pmatrix},\]
where
$$\begin{array}{ll}
\gamma (2b_1 - na_0)=0, & \qquad  \qquad \beta_3(b_1-2a_0)=0, \notag \\
\beta_k(b_1-(k-1)a_0)& =\frac k 2 a_1\sum\limits_{j=4}^{k}\beta_{j-1}\beta_{k-j+3}, \qquad 4 \leq k \leq n-1, \\
\beta_n(b_1-(n-1)a_0) & = - a_1\gamma + \frac n 2 a_1\sum\limits_{j=4}^{n}\beta_{j-1}\beta_{n-j+3}.\ \notag
\end{array}$$
\end{prop}

{\bf Case $F_{2}(0,0,\dots,0,0,1)$ and $n$-odd.}

\

From the relations of Proposition \ref{pr34}, it follows $b_1=\d\frac{n}{2}a_0$ and $a_1=0$. Therefore, the number of nil-independent outer derivations of the algebra $F_{2}(0,0,\dots,0,0,1)$ is equal to one. According to Theorem \ref{thm26}, we conclude that any solvable Leibniz algebra whose nilradical is $F_{2}(0,0,\dots,0,0,1)$ has dimension $n+2.$

\begin{thm} \label{thm35} Any solvable $(n+2)$-dimensional (the case of odd $n$) Leibniz algebra with nilradical $F_{2}(0,0,\dots,0,0,1)$ is isomorphic to the following algebra:
$$\ll_1: \left\{\begin{array}{lllll}
[e_0,e_0]=e_2,     &                 &[e_1,e_{1}]=e_n,           &[e_0,x]=e_0,             &\\[1mm]
[e_i,e_0]=e_{i+1}, & 2\leq i\leq n-1,&[x,e_1]=-\d\frac{n}{2} e_1,&[e_1,x]=\d\frac{n}{2}e_1.&
\end{array}\right.$$
\end{thm}

\begin{proof} Using Proposition \ref{pr34}, we have the products:
$$\begin{array}{lllll}
[e_0,e_0]=e_2,                      &                 &[e_1,e_{1}]=e_n,                    &[e_0,x]=a_0 e_0+\d\sum_{i=2}^n a_i e_i,      &     \\[1mm]
[e_i,e_0]=e_{i+1},                  & 2\leq i\leq n-1,&[x,e_1]=\d\sum_{i=0}^n \gamma_i e_i,&[e_1,x]=\frac{n}{2}a_0e_1+b_n e_n,           &      \\[1mm]
[x,e_0]=\d\sum_{i=0}^{n}\mu_i e_i,  &                  &                                   &[e_i,x]=ie_i+\d\sum_{j=i+1}^n a_{j+1-i} e_j, & 2\leq i\leq n,\\[1mm]
                                    &                   &                                  &[x,x]=\d\sum_{i=0}^n \delta_i e_i.&
\end{array}$$

Without loss of generality we can suppose $a_0=1$. It is easy to see that $<e_2,e_3,\dots,e_n>\subseteq Ann_r(\ll)$. Hence, $[x,e_i]=0$ for $2\leq i\leq n$.

Let us take the following transformation of basis:
$$e'_0=e_0+\d\sum_{i=2}^{n}A_i e_i,\quad e'_1=e_1,\quad e'_i=e_i+\d\sum_{k=i+1}^{n}A_{k-i+1} e_k,\ 2\leq i\leq n,\ x'=x$$
with $A_2=-a_2,\qquad A_i=\d\frac{1}{1-i}\left( a_i+\d\sum_{j=2}^{i-1}A_j a_{i-j+1}\right),\ \ 3\leq i\leq n.$
Then we obtain
$$\begin{array}{ll}
[x',e'_0]&=[x,e_0]=\d\sum_{i=0}^{n}\mu_i e_i,\\[3mm]
[e'_0,x']&=e_0+\d\sum_{i=2}^n a_ie_i+\d\sum_{i=2}^n A_i \left(ie_i+\d\sum_{j=i+1}^n a_{j+1-i} e_j\right)=e_0+\d\sum_{i=2}^n a_ie_i+A_2(2e_2+\d\sum_{j=3}^n a_{j-1} e_j)+\\[1mm]
&A_3(3e_3+\d\sum_{j=4}^n a_{j-2} e_j)+\cdots+A_{n-2}((n-2)e_{n-2}+\d\sum_{j=n-1}^n a_{j-n+3} e_j)+A_{n-1}((n-1)e_{n-1}+a_2 e_n)+\\[1mm]
&+A_n (n e_n)=e_0+(a_2+2A_2)e_2+\d\sum_{i=3}^n (i A_i+a_i+\d\sum_{j=2}^{i-1}A_j a_{i-j+1})e_i=e_0+\d\sum_{i=2}^{n} A_i e_i=e'_0.
\end{array}$$
Thus, we can assume $a_i=0$ for $2\leq i\leq n$.

Now, making the change $x'=x-\d\sum_{i=2}^{n-1}\mu_{i+1} e_i$ we obtain the family:

$$\begin{array}{lllll}
[e_0,e_0]=e_2,                          &                 &[e_1,e_{1}]=e_n,                     &[e_0,x]=e_0, &  \\[1mm]
[e_i,e_0]=e_{i+1},                      & 2\leq i\leq n-1,&[x,e_1]=\d\sum_{i=0}^n \gamma_i e_i, &[e_1,x]=\frac{n}{2}e_1+b_n e_n,&\\[1mm]
[x,e_0]=\mu_0 e_0+\mu_1 e_1+\mu_2 e_2,  &                   &                                   & [e_i,x]=ie_i, & 2\leq i\leq n,\\[1mm]
                                        &                   &                                   &[x,x]=\d\sum_{i=0}^n \delta_i e_i.&
\end{array}$$

By setting $e'_1=e_1-\frac{2}{n}b_n e_n$ we get $b'_n=0.$

The equalities
$$0=[e_0,[x,x]]=[e_1,[x,x]]=[e_0,[e_0,x]+[x,e_0]]=[e_1,[e_0,x]+[x,e_0]]=[e_1,[e_1,x]+[x,e_1]]$$
derive $$\delta_0=\delta_1=\mu_1=\gamma_0=0, \quad \mu_0=-1, \quad \gamma_1=-\frac{n}{2}.$$

Applying Leibniz identity for the triples $\{x,e_0,e_1\},$ $\{x,x,e_0\}$ and $\{x,x,e_1\},$
we conclude $$\gamma_i=0, \ \ 2\leq i\leq n,\quad \delta_i=0,\ \ 2\leq i\leq n-1, \quad \mu_2=0.$$

Finally, putting $x'=x-\d\frac{\delta_n}{n} e_n,$ we obtain the family of algebras $\ll_1.$
\end{proof}

{\bf Case $F_2(0,0,\dots,0,\beta_{\frac{n+2}{2}},0,\dots,0,0,1)$ and $n$-even.}

\

Similarly, we get $a_0=1, \ b_1=\frac{n}{2}, \ a_1=0$ and that any solvable Leibniz algebra whose nilradical is $F_2^1(0,0,\dots,0,\beta_{\frac{n+2}{2}},0,\dots,0,0,1)$ has dimension $n+2.$

\begin{thm}\label{thm36}
Any solvable Leibniz algebra with nilradical $F_{2}^1(0,0,\dots,0,\beta_{\frac{n+2}{2}},0,\dots,0,0,1)$ (the case of even $n$) is isomorphic to an algebra of the following family of algebras:
$$\ll_{2}^{\beta_{\frac{n+2}{2}}}: \left\{\small
\begin{array}{llllll}
[e_0,e_0]=e_{2},& &[e_0,e_1]=\beta_{\frac{n+2}{2}} e_{\frac{n+2}{2}},& &[e_0,x]=e_0,&\\[1mm]
[e_i,e_{0}]=e_{i+1}, & 2\leq i\leq n-1,&[e_1,e_{1}]=e_n,& &[e_1,x]=\frac{n}{2}e_1,&\\[1mm]
[x,e_0]=-e_0,& &[e_i,e_1]=\beta_{\frac{n+2}{2}}e_{\frac{n+2i}{2}},&2\leq i\leq \frac{n}{2},&[e_i,x]=i e_i, & 2\leq i\leq n.\\[1mm]
& &[x,e_1]=-\frac{n}{2} e_1-\beta_{\frac{n+2}{2}} e_{\frac{n}{2}},& & &
\end{array}\right.$$
\end{thm}

\begin{proof}
Using the previous arguments and Proposition \ref{pr34}, we obtain
the multiplications:
$$\small\begin{array}{llllll}
[e_0,e_0]=e_{2},                    &                  &[e_0,e_1]=\beta_{\frac{n+2}{2}} e_{\frac{n+2}{2}},&                        &[e_0,x]=e_0+\d\sum_{i=2}^n a_i e_i, &\\[1mm]
[e_i,e_{0}]=e_{i+1},                & 2\leq i\leq n-1,&[e_1,e_{1}]=e_n,                                    &                       &[e_1,x]=\frac{n}{2}e_1+b_n e_n,&\\[1mm]
[x,e_0]=\d\sum_{i=0}^{n}\mu_i e_i, &                  &[e_i,e_1]=\beta_{\frac{n+2}{2}}e_{\frac{n+2i}{2}},&2\leq i\leq \frac{n}{2},&[e_i,x]=i e_i+\d\sum_{j=i+1}^n a_{j+1-i} e_j, &2\leq  i\leq n,\\[1mm]
                                    &                   &[x,e_1]=\d\sum_{i=0}^n \gamma_i e_i,       &                               &[x,x]=\d\sum_{i=0}^n \delta_i e_i.&
\end{array}$$

Taking the following transformation of basis:
$$e'_0=e_0+\d\sum_{i=2}^{n}A_i e_i,\quad e'_1=e_1,\quad e'_i=e_i+\d\sum_{k=i+1}^{n}A_{k-i+1} e_k,\ 2\leq i\leq n,\quad x'=x$$
with $$A_2=-a_2,\quad A_i=\d\frac{1}{1-i}\left( a_i+\d\sum_{j=2}^{i-1}A_j a_{i-j+1}\right), \ 3\leq i\leq n,$$
we can assume that $a_i=0$ for $2\leq i\leq n.$

By setting $$x'=x-\d\sum_{i=2}^{n-1}\mu_{i+1} e_i, \quad e'_1=e_1-\frac{2}{n}b_n e_n$$ we reduce the above multiplication to the following one:
$$\small\begin{array}{lll}
[e_0,e_0]=e_{2},                             &[e_0,e_1]=\beta_{\frac{n+2}{2}} e_{\frac{n+2}{2}},                        &[e_0,x]=e_0, \\[1mm]
[e_i,e_{0}]=e_{i+1},   \quad 2\leq i\leq n-1,&[e_1,e_{1}]=e_n,                                                          &[e_1,x]=\frac{n}{2}e_1,\\[1mm]
[x,e_0]=\mu_0 e_0+\mu_1 e_1+\mu_2 e_2,       &[e_i,e_1]=\beta_{\frac{n+2}{2}}e_{\frac{n+2i}{2}},\ 2\leq i\leq \frac{n}{2},&[e_i,x]=i e_i,\ \ 2\leq i\leq n,\\[1mm]
                                             &[x,e_1]=\d\sum_{i=0}^n \gamma_i e_i,                                     &[x,x]=\d\sum_{i=0}^n \delta_i e_i.
\end{array}$$

From the equalities
$$[e_1,[x,x]]=[e_0,[x,x]]=[e_1,[e_0,x]+[x,e_0]]=
[e_0,[e_0,x]+[x,e_0]]=$$ $$=[e_1,[e_1,x]+[x,e_1]]=[e_0,[e_1,x]+[x,e_1]]=0$$
we derive
$$\delta_1=\delta_0=\mu_1=\gamma_0=0, \quad \mu_0=-1, \quad \gamma_1=-\frac{n}{2}$$

Applying the Leibniz identity for the triples $\{x,e_0,x\}, \ \{x,e_0,e_1\}$ and $\{x,x,e_1\},$ we obtain  $\mu_2=0,$ $\delta_i=0$ with $2\leq i\leq n-1,$
$\gamma_i=0$ with $2\leq i\leq n-1,$ $i\neq\frac{n}{2}$ and $\gamma_{\frac{n}{2}}=-\beta_{\frac{n+2}{2}}.$

The following change: $x'=x-\frac{\delta_n}{n}e_n$ deduce $\delta'_n=0.$  \end{proof}

{\bf Case $F_2^j(0,0,\dots,0,\underbrace{1}_j,0,\dots,0,0,0),$  $n\geq 4$ and $3\leq j\leq n$.}

\

Let us fix values $j$ and $j_0,$ $3\leq j, j_0\leq n$ such that $\beta_{j_0}$ is the first non zero parameter. Similar to the above cases and using Proposition \ref{pr34}, we derive
$$d(e_0)=\d\sum_{i=0}^n a_i e_i,\qquad d(e_1)=b_1e_1+b_ne_n$$
and using the restrictions on derivations we have:
$$\beta_{2j_0-2}(b_1-(2j_0-3)a_0)=0=\frac{2j_0-2}{2}a_1(\beta_{j_0}^2)\Rightarrow a_1=0,$$
$$2j_0-2\leq n\Rightarrow j_0\leq \lfloor\frac{n+2}{2}\rfloor$$
and $b_1=(j-1)a_0.$

Thus,
$$\left\{\begin{array}{ll}
a_1=0& \mbox{ if } 3\leq j\leq \lfloor\frac{n+2}{2}\rfloor\\
a_1& \mbox{ in other case}
\end{array}\right.$$

Similar to the previous cases, we can suppose $a_0=1$ and $b_1=(j_0-1).$ Therefore, the number of nil-independent outer derivations of the algebras
$F_2^{j_0}(0,0,\dots,0,\underbrace{1}_{j_0},0,\dots,0,0,0)$ is equal to one. We use $F_2^{j_0}$ to denote $F_2^{j_0}(0,0,\dots,0,\underbrace{1}_{j_0},0,\dots,0,0,0).$

\begin{thm}
Any solvable $(n+2)$-dimensional Leibniz algebra with nilradical $F_{2}^{j_0}$ with $3\leq j_0\leq n$ is isomorphic to the following algebra:

$$\ll_3^{j_0}: \left\{ \small\begin{array}{lll}
[e_0,e_0]=e_{2},                       &[e_0,e_1]=e_{j_0},                                &[e_0,x]=e_0,          \\[1mm]
[e_i,e_{0}]=e_{i+1}, \quad 2\leq i\leq n-1, &[e_i,e_1]=e_{j_0+i-1}, \quad  2\leq i\leq n-1-j_0,&[e_1,x]=(j_0-1)e_1,   \\[1mm]
[x,e_0]=-e_0,                      &[x,e_1]=-(j_0-1)e_1-e_{j_0-1},                    &[e_i,x]=i e_i, \quad 2\leq i\leq n.
\end{array}\right. $$
\end{thm}

\begin{proof} The proof is carried out by applying arguments used in Theorems \ref{thm35} and \ref{thm36}.
\end{proof}

\

\subsection{Solvable Leibniz algebras with nilradical a non-characteristically nilpotent algebra of the family $F_3(\theta_1,\theta_2,\theta_3)$.}

\

\

Let $\ll$ be a filiform Leibniz algebra from the family $F_3(\theta_1,\theta_2,\theta_3)$ and let $\{e_0,e_1,\dots,e_n\}$ be a basis. The following proposition describes
the derivations of such algebras.

\begin{prop}\label{pr38}\cite{abror} A derivation $d$ of a filiform Leibniz algebra of the family $F_3(\theta_1,\theta_2,\theta_3)$ have the following form:
$$\begin{array}{l}
d(e_0)=\sum\limits_{i=0}^na_ie_i, \quad d(e_1)=\sum\limits_{i=1}^nb_ie_i,\\
d(e_i)=((i-1)a_0+b_1)e_i+\sum\limits_{j=i+1}^{n-1}b_{j-i+1}e_j+(b_{n-i+1}+(-1)^{i-1}\alpha a_{n-i+1})e_n,\\
d(e_n)=((n-1)a_0+b_1+\alpha a_1)e_n
\end{array}$$
with the following restrictions:
$$\begin{array}{l}
\theta_1((n-3)a_0+b_1)=a_1\theta_2,\\
2a_1\theta_3=(n-2)a_0\theta_2,\\
\theta_3((n-1)a_0-b_1)=0.
\end{array}$$
\end{prop}

We consider the case of solvable Leibniz algebras with non-Lie filiform nilradical of the family $F_3(\theta_1,\theta_2,\theta_3)$ bellow.

\begin{thm}
 There is not any solvable Leibniz algebra whose nilradical is a non-characteristically filiform non-Lie algebra of the family $F_3(\theta_1,\theta_2,\theta_3).$
\end{thm}

\begin{proof} Let us consider firstly the {\bf Case $F_3^1(1,0,0)$.}
Proposition \ref{pr38} leads to $b_1=(3-n)a_0$ and we can suppose $a_0=1$ making the change $x'=\frac{1}{a_0}x.$ Thus, we have the following products:
$$\begin{array}{ll}
[e_i,e_0]=-[e_0,e_i]=e_{i+1}, & 1\leq i\leq n-1,\\{}
[e_0,e_0]=e_n,&\\{}
[e_i,e_{n-i}]=-[e_{n-i},e_i]=(-1)^i\alpha e_n, & 1\leq i\leq n-1,\\{}
[e_0,x]=e_0+\sum\limits_{i=1}^na_ie_i,&\\{}
[e_1,x]=(3-n)e_1+\sum\limits_{i=2}^nb_ie_i,&\\{}
[e_i,x]=((i-1)+b_1)e_i+\sum\limits_{j=i+1}^{n-1}b_{j-i+1}e_j+(b_{n-i+1}+(-1)^{i-1}\alpha a_{n-i+1})e_n,&\\{}
[e_n,x]=((n-1)+b_1+\alpha a_1)e_n,&\\{}
[x,e_0]=\sum\limits_{i=0}^n\beta_ie_i,& \\{}
[x,e_1]=\sum\limits_{i=0}^n\gamma_ie_i,& \\{}
[x,x]=\sum\limits_{i=0}^n\delta_ie_i.&
\end{array}$$

It is easy to check that by the change $x'=x+\sum\limits_{i=1}^{n-1}a_{i+1}e_i$ we can suppose $a_i=0$ with $2\leq i\leq n.$

Applying the Leibniz identity for the elements $\{e_2,e_{n-2},x\}$ and $\{e_0,e_0,x\},$ we get a contradiction with
$\alpha=1.$ Thus, $\alpha=0.$

Taking the following basis transformation:
$$e_0'=e_0+\frac{a_1}{n-2}e_1, \ \ e_i'=e_i, \ 1\leq i\leq n, \ \ x'=x+\frac{a_1}{n-2}\sum\limits_{i=1}^{n-1}b_{i+1}e_i,$$
we can suppose $a_1=0.$

Putting
$$e_0'=e_0, \quad e_i'=e_i+\sum\limits_{j=i+1}^nA_{j-i+1}e_j, \ 1\leq i\leq n,$$
with
$$A_2=-b_2, \ A_i=\frac{1}{1-i}(b_i+\sum\limits_{j=2}^{i-1}A_jb_{i-j+1}), \ 3\leq i\leq n,$$
one can assume $b_i=0$ for $2\leq i\leq n.$

Let us resume the products of the family
$$\begin{array}{lll}
[e_i,e_0]=-[e_0,e_i]=e_{i+1}, & 1\leq i\leq n-1,&\\{}
[e_0,e_0]=e_n,& &\\{}
[e_0,x]=e_0,& &[x,e_0]=\sum\limits_{i=0}^n\beta_ie_i, \\{}
[e_i,x]=(i-n+2)e_i,&1\leq i\leq n,&[x,e_1]=\sum\limits_{i=0}^n\gamma_ie_i,\\{}
[x,x]=\sum\limits_{i=0}^n\delta_ie_i.& &
\end{array}$$

Considering Leibniz identity for the elements of the form $$\{e_0,x,e_0\}, \quad \{e_0,x,e_1\}, \quad \{e_1,x,e_1\}, \quad \{e_1,x,e_0\},$$
we obtain that  $$\beta_0=-1, \ \beta_i=0, \ 1\leq i\leq n-2, \qquad \gamma_0=0, \ \gamma_1=n-3, \ \gamma_i=0, \ 2\leq i\leq n-1.$$

Using the induction method, we get $[x,e_i]=-(i-n+2)e_i,$ with $ 2\leq i\leq n.$

From the equality $0=[x,[e_0,e_0]],$ we have $[x,e_n]=0,$ but this is a contradiction, because $[x,e_n]=-2e_n.$
Therefore, there are not any solvable Leibniz algebras with $F_3^1(1,0,0)$-nilradical.

\

Similar study of {\bf Case $F_3^2(0,1,0)$} and {\bf Case $F_3^2(0,0,1)$} leads to non-existence of solvable Leibniz algebras with nilradicals $F_3^2(0,1,0)$ and $F_3^2(0,0,1).$
\end{proof}

\section{Solvable Leibniz algebras with filiform Lie nilradical}

In this section we study solvable Leibniz algebras whose nilradical is a filiform Lie algebra. Since we consider non-nilpotent solvable Leibniz algebras, it is sufficient to consider non-characteristically nilpotent filiform Lie nilradicals. In this section we restrict ourselves to the study of the families $A_{n+1}^r(\alpha_1,\dots,\alpha_t),$  $B_{n+1}^r(\alpha_1,\dots,\alpha_t)$ of Theorem \ref{yu}, because the other two algebras of Theorem \ref{yu} have been already studied in \cite{SolGFil}.


\

{\bf Case $A_{n+1}^r(\alpha_1,\dots,\alpha_t)$ with $1\leq r\leq n-3$ and $t=\lfloor\frac{n-r-1}{2}\rfloor.$}

\

Below we present some description of the derivations of the family of algebras $A_{n+1}^r(\alpha_1,\dots,\alpha_t).$

\begin{prop}\label{pr41}
Any derivation of a filiform Lie algebra of the family $A_{n+1}^r(\alpha_1,\dots,\alpha_t)$ with $1\leq r\leq n-3$ and $t=\lfloor\frac{n-r-1}{2}\rfloor$
has the form:
$$\begin{array}{ll}
d(e_0)=\sum\limits_{i=0}^na_ie_i,&\\
d(e_1)=(1+r)a_0e_1+\sum\limits_{i=2}^nb_ie_i,&\\
d(e_i)=(i+r)a_0e_i+\sum\limits_{j=i+1}^n(*)e_j,& 2\leq
i\leq n,
\end{array}$$
where $\{e_0,e_1,\dots,e_n\}$ is a basis of the family $A_{n+1}^r.$
\end{prop}

\begin{proof}
Let us denote
$$d(e_0)=\sum\limits_{i=0}^na_ie_i, \quad d(e_1)=\sum\limits_{i=0}^nb_ie_i.$$

Using the induction method and the properties of derivation, we establish
$$d(e_i)=((i-1)a_0+b_1)e_i+\sum\limits_{j=i+1}^n(*)e_j,\quad 2\leq i\leq n.$$

From the equalities
$$0=d([e_1,e_{n-1}])=[d(e_1),e_{n-1}]+[e_1,d(e_{n-1})]=b_0e_n$$
we get $b_0=0.$

\begin{itemize}
\item If $\alpha_1\neq 0,$ then from the equality $d([e_1,e_2])=d(\alpha_1e_{r+3})$
we have $b_1=(r+1)a_0.$

\item If $\alpha_1=0,$ then there exists $i,$ $2\leq i\leq t$ such that $\alpha_i\neq 0.$
The equality $d([e_i,e_{i+1}])=d(\alpha_i e_{2i+1+r})$ implies $b_1=(r+1)a_0.$
\end{itemize}

Thus, we obtain the form of derivation which is asserted in the proposition.
\end{proof}

From Proposition \ref{pr41}, we conclude that the number of nil-independent outer derivations of the algebra $A_{n+1}^r(\alpha_1,\dots,\alpha_t)$ is equal to one. Consequently, according to Theorem \ref{thm26}, any solvable Leibniz algebra whose nilradical is $A_{n+1}^r(\alpha_1,\dots,\alpha_t)$ has dimension $n+2.$

In this section we use similar arguments to those from the above section.

\begin{thm}\label{A}
 Any solvable Leibniz algebra with nilradical $A_{n+1}^r(\alpha_1,\dots,\alpha_t)$ is isomorphic to the following family of Lie algebras:
 $$\small\begin{array}{ll}
[e_0,e_i]&=e_{i+1},\quad 1\leq i\leq n-1,\\{}
[e_i,e_j]&=\left(\sum\limits_{k=i}^t(-1)^{k-i}\alpha_k
\left (\begin{array}{c}
j-k-1\\
k-i
\end{array} \right )\right) e_{i+j+r},
 \quad 1\leq i<j\leq n-2, \ i+j+r\leq n,\\{}
[e_0,x]&=e_0+a_1e_1,\\{}
[e_1,x]&=(1+r)e_1+\sum\limits_{i=2}^nb_ie_i,\\{}
[e_2,x]&=(2+r)e_2+\sum\limits_{i=3}^nb_{i-1}e_i,\\{}
[e_i,x]&=(i+r)e_i+\sum\limits_{j=i+1}^{i-1+r} b_{j-i+1} e_j+\left(b_{1+r}+a_1\left(\sum\limits_{k=2}^{i-1}\left (\sum\limits_{s=1}^t (-1)^{s-1}\alpha_s
\left (\begin{array}{c}
k-s-1\\
s-1
\end{array} \right )\right)
\right) \right)e_{i+r}+\\
&+\sum\limits_{j=i+1+r}^{n} b_{j-i+1} e_j
\ \ \ \mbox{ with }\ 3\leq i\leq n-r,\\{}
[e_i,x]&=(i+r)e_i+\sum\limits_{j=i+1}^{n} b_{j-i+1} e_j,\quad n-r+1\leq i\leq n.
\end{array}$$
\end{thm}

\begin{proof}
From Proposition \ref{pr41} we have the following products:
$$\begin{array}{ll}
[e_0,e_i]=e_{i+1},& 1\leq i\leq n-1,\\{}
[e_i,e_j]=\left(\sum\limits_{k=i}^t(-1)^{k-i}\alpha_k \left (\begin{array}{c}
j-k-1\\
k-i
\end{array} \right )\right)e_{i+j+r}, & 1\leq i<j\leq n-2, \ i+j+r\leq n,\\{}
[e_0,x]=\sum\limits_{i=0}^na_ie_i,&\\{}
[e_1,x]=(1+r)a_0e_1+\sum\limits_{i=2}^nb_ie_i,&\\{}
[e_i,x]=(i+r)a_0e_i+\sum\limits_{j=i+1}^n(*)e_j,& 2\leq
i\leq n.
\end{array}$$

Since $a_0\neq0$, by scaling of basis element $x$, we can suppose $a_0=1.$
The change $x'=x-\d\sum\limits_{i=1}^{n-1}a_{i+1}e_i$ admits to suppose $a_i=0$ for $2\leq i\leq n.$

Thus, we have
$$\begin{array}{ll}
[e_0,e_i]=e_{i+1},& 1\leq i\leq n-1,\\{}
[e_i,e_j]=\left(\sum\limits_{k=i}^t(-1)^{k-i}\alpha_k \left (\begin{array}{c}
j-k-1\\
k-i
\end{array} \right )\right)e_{i+j+r}, & 1\leq i<j\leq n-2, \ i+j+r\leq n,\\{}
[e_0,x]=e_0+a_1e_1,&\\{}
[e_1,x]=(1+r)e_1+\sum\limits_{i=2}^nb_ie_i,&\\{}
\end{array}$$

Using Jacobi identity and the induction method, we obtain the expressions of the products $[e_i,x]$ with $2\leq i\leq n,$ which complete the proof of theorem. \end{proof}

\begin{prop}\label{A.Leibniz}
 There are not any solvable non-Lie Leibniz algebras with nilradical $A_{n+1}^r(\alpha_1,\dots,\alpha_t).$
\end{prop}

\begin{proof}
From Proposition \ref{pr41} and applying similar arguments to the ones used in the proof of Theorem \ref{A}, we have
 $$\small\begin{array}{ll}
[e_0,e_i]&=e_{i+1},\quad 1\leq i\leq n-1,\\{}
[e_i,e_j]&=\left(\sum\limits_{k=i}^t(-1)^{k-i}\alpha_k
\left (\begin{array}{c}
j-k-1\\
k-i
\end{array} \right )\right) e_{i+j+r},
 \quad 1\leq i<j\leq n-2, \ i+j+r\leq n,\\{}
[e_0,x]&=e_0+a_1e_1,\\{}
[e_1,x]&=(1+r)e_1+\sum\limits_{i=2}^nb_ie_i, \\{}
[e_2,x]&=(2+r)e_2+\sum\limits_{i=3}^nb_{i-1}e_i, \\{}
[e_i,x]&=(i+r)e_i+\sum\limits_{j=i+1}^{i-1+r} b_{j-i+1} e_j+\left(b_{1+r}+a_1\left(\sum\limits_{k=2}^{i-1}\left (\sum\limits_{s=1}^t (-1)^{s-1}\alpha_s
\left (\begin{array}{c}
k-s-1\\
s-1
\end{array} \right )\right)
\right) \right)e_{i+r}+\\
&+\sum\limits_{j=i+1+r}^{n} b_{j-i+1} e_j\ \ \ \mbox{ with }\ \ 3\leq i\leq n-r,\\{}
[e_i,x]&=(i+r)e_i+\sum\limits_{j=i+1}^{n} b_{j-i+1} e_j,\quad n-r+1\leq i\leq n.
\end{array}$$

Let us denote
$$[x,e_0]=\displaystyle\sum\limits_{i=0}^n\beta_ie_i,\qquad [x,e_1]=\sum\limits_{i=0}^n\gamma_ie_i,\qquad [x,x]=\sum\limits_{i=0}^n\delta_ie_i.$$

Applying the Leibniz identity for the elements $$\{e_0,x,x\}, \quad \{e_0,x,e_0\}, \quad \{e_1,x,e_0\}$$ we derive
$$\delta_i=0, \ 0\leq i\leq n-1,\quad \beta_i=0, \ 2\leq i\leq n-1, \quad \beta_1=-a_1, \quad \beta_0=-1.$$

Verifying the Leibniz identity for the products and using the induction method, we compute the following products:
$$\small\begin{array}{ll}
[x,e_2]&=-(1-\gamma_1)e_2+\sum\limits_{i=3}^n \gamma_{i-1}e_i,\\{}
[x,e_i]&=-(i-1-\gamma_1)e_i+\sum\limits_{j=i+1}^{i-1+r} \gamma_{j-i+1} e_j+\left(\gamma_{1+r}+a_1\left(\sum\limits_{k=2}^{i-1}\left(\sum\limits_{s=1}^t (-1)^{s-1}\alpha_s
\left (\begin{array}{c}
k-s-1\\
s-1
\end{array} \right )
\right)\right) \right)e_{i+r}+\\
&+\sum\limits_{j=i+1+r}^{n} \gamma_{j-i+1} e_j, \ \ 3\leq i\leq n-r,\\{}
[x,e_i]&=-(i-1-\gamma_1)e_i+\sum\limits_{j=i+1}^{n} \gamma_{j-i+1} e_j,\quad n-r+1\leq i\leq n.
\end{array}$$

Since $[e_1,x]+[x,e_1]\in Ann_r(\ll)$, we get

$$[e_{n-1},[e_1,x]+[x,e_1]]=[e_{0},[e_1,x]+[x,e_1]]=[x,[e_1,x]+[x,e_1]]=0,$$ from which we obtain
$$\gamma_0=0, \quad \gamma_1=-1-r, \quad \gamma_i=-b_i, \ 2\leq i\leq n-1, \quad \mu_n=-b_n.$$

Similarly, from $[x,[e_0,x]+[x,e_0]]=0$ we conclude $ \beta_n=0.$

Considering the equality $0=[x,[x,x]]$ we get $\delta_n=0.$
Thus, we obtain a Lie algebra.
\end{proof}

\


{\bf Case $B_{n+1}^r(\alpha_1,\dots,\alpha_t)$ with $1\leq r\leq n-4$ and $t=\lfloor\frac{n-r-2}{2}\rfloor$.}

\

The study of this case is similar to previous case.

\begin{prop}\label{pr44}
Any derivation of a filiform Lie algebra of the family $B_{n+1}^r(\alpha_1,\dots,\alpha_t)$ with $1\leq r\leq n-4$ and $t=\lfloor\frac{n-r-2}{2}\rfloor$
has the form:
$$\begin{array}{ll}
d(e_0)=a_0e_0+\sum\limits_{i=2}^na_ie_i,&\\
d(e_1)=(1+r)a_0e_1+\sum\limits_{i=2}^nb_ie_i,&\\
d(e_i)=(i+r)a_0e_i+\sum\limits_{j=i+1}^n(*)e_j, & 2\leq i\leq n-1,\\
d(e_n)=(n+2r)a_0e_n,&
\end{array}$$
where $\{e_0,e_1,\dots,e_n\}$ is a basis of the family $B_{n+1}^r.$
\end{prop}

\begin{proof}
In an analogous way to the proof of Proposition \ref{pr41}.
\end{proof}

From Proposition \ref{pr44} and Theorem \ref{thm26}, we conclude that any solvable Leibniz algebra whose nilradical is $B_{n+1}^r(\alpha_1,\dots,\alpha_t)$ has dimension $n+2.$

\begin{thm}\label{B}
 Any solvable Leibniz algebra with nilradical $B_{n+1}^r(\alpha_1,\dots,\alpha_t)$ is isomorphic to an algebra of the following family of Lie algebras:
 $$\begin{array}{ll}
[e_0,e_i]=e_{i+1},& 1\leq i\leq n-2,\\{}
[e_i,e_{n-i}]=(-1)^ie_n, & 1\leq i\leq n-1,\\{}
[e_i,e_j]=\left(\sum\limits_{k=i}^t(-1)^{k-i}\alpha_k
\left (\begin{array}{c}
j-k-1\\
k-i
\end{array} \right )
\right)e_{i+j+r}, & 1\leq i,j\leq n-1, \ i+j+r\leq n-1,\\{}
[e_0,x]=e_0,&\\{}
[e_1,x]=(1+r)e_1+\sum\limits_{i=2}^{n-1}b_ie_i,&\\{}
[e_i,x]=(i+r)e_i+\sum\limits_{j=i+1}^{n-1}b_{j-i+1}e_j, \ 2\leq i\leq n-1,&\\{}
[e_n,x]=(n+2r)e_n.&
\end{array}$$
\end{thm}

\begin{proof}
The proof is similar to Theorem \ref{A}.
\end{proof}

\begin{prop} \label{AA}
 There are not any solvable non-Lie Leibniz algebras with nilradical $B_{n+1}^r(\alpha_1,\dots,\alpha_t).$
\end{prop}

\begin{proof}
The proof is similar to Proposition \ref{A.Leibniz}.
\end{proof}

From Propositions \ref{A.Leibniz} and \ref{AA}, we resume that any solvable Leibniz algebra whose nilradical is either $A_{n+1}^r(\alpha_1,\dots,\alpha_t)$ or $B_{n+1}^r(\alpha_1,\dots,\alpha_t)$ is Lie algebra.


\

\noindent\textbf{Conjecture.}
\begin{itemize}
\item[$(i)$] By the following transformation basis of an algebra of the family from Theorem \ref{A}:
$$e_0'=e_0, \quad e_1'=e_1+\sum\limits_{i=2}^{n-1}A_{i}e_i, \quad e_i'=e_i+\sum\limits_{j=i+1}^nA_{j-i+1}e_j, \quad 2\leq i\leq n, \eqno (*)$$
with $$A_2=-b_2, \quad A_i=\frac{1}{1-i}(b_i+\sum\limits_{j=2}^{i-1}A_jb_{i-j+1}), \quad 3\leq i\leq n,$$
we can eliminate the parameters $b_i=0$ with $2\leq i\leq n$

\

\item[$(ii)$] By the transformation of basis (*) an algebra of the family from Theorem \ref{B} with
$$A_2=-b_2, \quad A_3=\frac{b^2_2}{2}, \quad A_{2k}=\frac{1}{1-2k}(b_{2k}+\sum\limits_{j=2}^{k}A_{2j-1}b_{2k-2j+2}), \quad 2\leq k\leq \frac{n-1}{2},$$ $$\quad A_{2k+1}=-\frac{1}{2k}(\sum\limits_{j=1}^{k}A_{2j}b_{2k-2j+2}), \quad 2\leq k\leq \frac{n-3}{2},$$
we can eliminate the parameters $b_i=0$ with $2\leq i\leq n.$
\end{itemize}

The correctness of Conjecture was checked for fixed low dimensions by program Mathematica. However, we could not to generalize the calculation for low dimensions for any finite dimension because the great number of complex calculations needed.

\section*{ Acknowledgements}
 The first author would like to thank of
the Institute of Ma\-the\-ma\-tics (Uzbekistan) for their hospitality.
The second author was supported by the Grant (RGA) No:11-018 RG/Math/AS\_I--UNESCO FR: 3240262715.

\end{document}